\documentclass[10pt]{article}
\usepackage[utf8]{inputenc}
\usepackage[english]{babel}
\usepackage{amsmath}
\usepackage{authblk}
\usepackage{amsfonts}
\usepackage{makeidx}
\usepackage[dvips]{graphicx}
\usepackage{pstricks}
\usepackage{pstricks-add}
\usepackage{marvosym}
\usepackage{pgf,tikz}
\DeclareGraphicsExtensions{.bmp,.png,.pdf,.jpg}
\usepackage{makeidx}
\usetikzlibrary{arrows}
\usepackage{theorem}
\usepackage{ mathrsfs }
\usepackage{comment}
\usepackage{amssymb}
\usepackage{ dsfont }
\usepackage{color}
\usepackage{fancyhdr}
\usepackage[margin=1.6in]{geometry}
\usepackage{fancybox}
\usepackage{makeidx}
\usepackage[explicit]{titlesec}
\usepackage{lipsum}
\usepackage{hyperref}

\newtheorem{teo}{Theorem}

\newtheorem{defi}{Definition}

\newtheorem{prop}{Proposition}
\newtheorem{cor}{Corolary}

\newenvironment{proof}{\noindent{\textbf{Proof}}}{\begin{flushright}
$\blacksquare$
\end{flushright}}
\newenvironment{teorem}[1][Theorem]{\begin{trivlist}
\item[\hskip \labelsep {\bfseries #1}]}{\end{trivlist}}

\title{{Rosenthal compacta that are premetric of finite degree}}
\author{Antonio Avil\' es, Alejandro Poveda and Stevo Todorcevic}
\date{}
\begin{document}
\maketitle
\begin{abstract}
We show that if a separable Rosenthal compactum $K$ is a continuous $n$-to-one preimage of a metric compactum, but it is not a continuous $n-1$-to-one preimage, then $K$ contains a closed subset homeomorphic to either the $n-$Split interval $S_n(I)$ or the Alexandroff $n-$plicate $D_n(2^\mathbb{N})$. This generalizes a result of the third author that corresponds to the case $n=2$.
\end{abstract}
\section{Introduction}
A compact space $K$ is a Rosenthal compactum if it is homeomorphic to a compact subset of $\boldsymbol{\mathcal{B}_1}(X)$, the space of real-valued functions of the first Baire class on a Polish space $X$ endowed with the topology of pointwise convergence. This is a well studied class originated in relation with the study of separable Banach spaces without copies of $\ell^1$ \cite{Ros} \cite{BFT} \cite{God} \cite{HMO}.
\newline

In the paper \cite{Tod}, three critical Rosenthal compacta are identified: The Split interval $S(I)$, the Alexandroff duplicate of the Cantor set $D(2^\mathbb{N})$, and the one-point compactification of a discrete set of size continuum $A(D)$. The definition of $S(I)$ and $D(2^\mathbb{N})$ are recalled in Section \ref{sec::section3}. One key property of these two compact spaces is that they are \textit{premetric compacta of degree at most two}. A compact space $K$ is a premetric compactum of degree at most two if there exists a continuous surjection $f:K\longrightarrow M$ onto a metric compactum $M$ such that $\vert f^{-1}(x)\vert\leq 2$ for all $x\in M$. It is proven in \cite{Tod} that a separable Rosenthal compactum which does not contain discrete subspaces of size continuum must be premetric of degree at most two. Another result is the following:
\begin{teorem}
\textbf{(S. Todorcevic)}
\textit{
If a separable Rosenthal compactum $K$ is a premetric compactum of degree at most two, then at least one of the following alternatives holds
\begin{enumerate}
\item $K$ is metric.
\item $K$ contains a homeomorphic copy of $S(I)$.
\item $K$ contains a homeomorphic copy of $D(2^\mathbb{N})$.
\end{enumerate}
}
\end{teorem}
If we are given a natural number $n$, we can say, more generally, that a compact space $K$ is a premetric compactum of degree at most $n$ if there exists a continuous surjection $f: K\longrightarrow M$ onto a metric compactum such that $\vert f^{-1}(x)\vert \leq n$ for all $x\in M$. For $n=1$ we get the class of metric compacta. In this work, we introduce $n-$ dimensional versions $S_n(I)$ and $D_n(2^\mathbb{N})$ of the Split interval and the Alexandroff duplicate, and prove the following generalization of the previous theorem:
\begin{teorem}
\textit{
Fix a natural number $n\geq 2$. If a separable Rosenthal compactum $K$ is a premetric compactum of degree at most $n$, then one of the following alternatives holds:
\begin{enumerate}
\item $K$ is a premetric compactum of degree at most $n-1$.
\item $K$ contains a homeomorphic copy of $S_n(I)$.
\item $K$ contains a homeomorphic copy of $D_n(2^\mathbb{N})$.
\end{enumerate}
}
\end{teorem}
Section \ref{sec::section2} contains some preliminary results, in Section \ref{sec::section3} we introduce the spaces $S_n(I)$ and $D_n(2^\mathbb{N})$, and Section \ref{sec::section4} contains the proof of the main result, that mimics \cite[Section E]{Tod} with some adaptations needed for the new multidimensional setting.

\section{Preliminaries}\label{sec::section2}
In this section we recall some results on descriptive set theory, general topology and Ramsey theory that we need throughout the paper. Fore further details we refer to \cite{Kec}, \cite{Deb}, \cite{Eng} and \cite{TodRam}.

We denote by $\mathbb{N}$ the set of natural numbers and identify each natural number $n$ with the set of its predecessors $\{0,\dots, n-1\}$. Given a nonempty set $X$ let us consider $X^{<\mathbb{N}}$ the set of finite sequences on $X$.
For every finite sequence $s=(s_0,\dots, s_{n-1})$ we denote by $length(s)$ the natural number $n$, the domain of $s$. If $n\leq length(s)$ by $s\upharpoonright n$ we will refer to the sequence $(s_0,\dots, s_{n-1})$, the restriction of $s$ to its first $n$ coordinates. Given $s,t\in X^{<\mathbb{N}}$ we say that $s$ is an initial segment of $t$ or $t$ is an extension of $s$ ($s\preceq t$) if there exists a natural number $m\leq length(t)$ such that $s=t\upharpoonright m$. If $length(s)<length(t)$ we say that this extension is proper. Otherwise, if $s\not\preceq t$ and $t\not\preceq s$  then we say that $s$ and $t$ are incomparable.
Given two sequences $s=(s_0,\dots, s_{n-1})$ and $t=(t_0,\dots, t_{k-1})$ we denote by $s\frown t$ the sequence $(s_0,\dots, s_{n-1},t_0,\dots, t_{k-1})$.

We say that a subset $T$ of $X^{<\mathbb{N}}$ is a tree on $X$ if it is nonempty and closed under initial segments. That is, if given $s\in T$ and $t\in X^{<\mathbb{N}}$ with $t\preceq s$ then $t\in T$. We will refer to the elements of $T$ as nodes of the tree $T$.

We say that a sequence $x\in X^{\mathbb{N}}$ is a branch of the tree $T$ if $x\upharpoonright n\in T$ for every $n\in\mathbb{N}$. When a given node $s\in T$ is a initial segment of a branch $x$ we write $s\subset x$ or $x\in [s]$. Finally, by $[T]$ we denote the set of all branches of $T$.

Given two nonempty sets $X,Y$ there is a natural way to build a tree $T$ on the product $X\times Y$. In such case, we restrict ourselves to consider nodes  of the form $s = (t,u)\in T$ such that $t\in X^{<\mathbb{N}}$, $u\in Y^{<\mathbb{N}}$ and $length(t)=length(u)$. A a couple $(t,u)$ will be an extension of another $(t',u')$ if and only if $t'$ and $u'$ are initial segments of $t$ and $u$ respectively. Finally, it is easy to prove that the set of all branches of the tree $T$ is
$$[T]=\{(u,t)\in X^{\mathbb{N}}\times Y^{\mathbb{N}}:\,(u\upharpoonright n,\,t\upharpoonright n)\in T,\; \forall n\in \mathbb{N}\}.$$
Recall that every product $X^\mathbb{N}$ can be endowed in a natural way with the product topology taking the discrete topology as the topology of the set $X$. An open basis of the topology of $X^\mathbb{N}$ is given by the sets $\{x\in X^\mathbb{N}: s\subset x\}$ where $s\in X^{<\mathbb{N}}$.

Remember that a topological space is said to be Polish if it is separable and completely metrizable; that is, there exists a complete metric compatible with its topology. If $X$ is a countable set then $X^{\mathbb{N}}$ is a Polish space endowed with its product topology. Two significant examples of that kind of Polish spaces are the Cantor space $2^\mathbb{N}$ and the Baire space $\mathbb{N}^\mathbb{N}$. It is well known, for example, that every perfect Polish space contains copies of $2^\mathbb{N}$ and hence it has cardinality continuum. 
The following notion is often used when building copies of $2^\mathbb{N}$ inside another space:
\begin{defi}[Cantor Scheme]\label{esquemadecantor}
Given a set $X$ we say that a family $\{\mathcal{A}_s\}_{s\in 2^{<\mathbb{N}}}$ of subsets of $X$ is a Cantor scheme over $X$ if it the following conditions are satisfied:
\begin{enumerate}
\item $\mathcal{A}_{s\frown 0}\cap \mathcal{A}_{s\frown 1}=\emptyset$.
\item $\mathcal{A}_{s\frown i}\subseteq \mathcal{A}_{s}$ for all $s\in 2^{<\mathbb{N}},\, i\in\{0,1\}$.
\end{enumerate} 
\end{defi}

Given any totally ordered set $(X,\preceq)$ we can endow $X$ with the topology generated by its $\prec-$rays $\{x\in X: x\prec x_0\}$ and $\{x\in X:x\succ x_0\}$ where $x_0\in X$. This topology is called \textit{the order topology induced by $\preceq$}. Recall that given two partially ordered sets $(X,\leq)$ and $(Y,\preceq)$ the lexicographical order on $X\times Y$ is defined by the following condition
$$(x_1,y_1)\leq_{lex}(x_2,y_2)\;\;\Longleftrightarrow\;\;\text{if}\;\; (x_1 < x_2)\;\;\text{or}\;\; (x_1=x_2,\;\;\text{and}\;\; y_1\preceq y_2).$$
A special property shared by every space whose topology arises from an order is the following
\begin{teo}\label{hereditarily}
For a topological space $X$ endowed with an order topology it is equivalent to be separable and hereditarily separable.
\end{teo}

\begin{proof} If $D$ is a countable dense subset of $X$, and $A$ is a subset of $X$, then for every interval $I$ (open or closed on the right or on the left) whose extremes belong to $D$, choose an element $a(I) \in A\cap I$ whenever $A\cap I\neq \emptyset$. The elements $a(I)$ form a countable dense subset of $A$.\end{proof}

A real-valued function $f$ is of the \textit{first Baire class} on a Polish space $X$, if it is the pointwise limit of real-valued continuous functions on $X$. That is, $f$ is a first Baire class function if there are continuous functions $\{f_n\}_{n\in\mathbb{N}}$ on $X$ such that $f(x)=\lim_n f_n(x)$ for every $x\in X$. This class of functions is usually denoted by $\boldsymbol{\mathcal{B}_1}(X)$. 

\begin{defi}[Rosenthal compactum]
We say that a topological space $K$ is a Rosenthal compactum if it is homeomorphic to a pointwise compact subspace of $\boldsymbol{\mathcal{B}_1}(\mathbb{N}^\mathbb{N})$.
\end{defi}
If $X$ is an arbitrary Polish space, then every pointwise compact subset of $\boldsymbol{\mathcal{B}_1}(X)$ is a Rosenthal compactum, because $X$ is a continuous image of $\mathbb{N}^\mathbb{N}$.
The class of Rosenthal compact spaces constitutes a generalization of the class of compact metrizable spaces and many of the properties of metrizable compacta hold true as well for Rosenthal compacta. One example is Bourgain-Fremlin-Talagrand's result on the Fréchet-Uryshon property of Rosenthal compacta \cite{BFT}, meaning that every point in the closure of a set is the limit of a convergent sequence from that set. Another remarkable result on the same line states that every Rosenthal compactum contains a dense metrizable subspace \cite{Tod}. We are specially interested in the separable Rosenthal compacta. If we have countably many Borel functions on a Polish space $X$, we can add countably many open sets to the topology of $X$ to get a larger Polish topology where all those functions are continuous. From this observation one gets that:

\begin{prop}\label{separableK}
Every separable Rosenthal compactum is homeomorphic to the pointwise closure in $\boldsymbol{\mathcal{B}_1}(\mathbb{N}^\mathbb{N})$ of a sequence of continuous functions.
\end{prop}

Recall that a Rosenthal compactum $K$ is premetric of degree at most $n$ if there exists a continuous surjection $f: K\longrightarrow M$ onto a metric compactum $M$ such that $|f^{-1}(x)|\leq n$ for all $x\in M$. Some classical examples of premetric compacta of degree at most $n$ are the compact metric spaces ($n=1$) as well as the Split interval $S(I)$ and the Alexandroff duplicate of the Cantor space $D(2^\mathbb{N})$ ($n=2$). In this regard, in this paper we present the $n-$dimensional versions $S_n(I)$ and $D_n(2^\mathbb{N})$ of $S(I)$ and $D(2^\mathbb{N})$ respectively, and we prove they are premetric compacta of degree at most $n$. In Section \ref{sec::section3} the definition of $S_n(I)$ and $D_n(2^\mathbb{N})$ will be presented along with their classical versions for $n=2$. There are also classical examples of Rosenthal compacta which are not premetric compacta of any degree. One of theses examples is the Alexandroff compactification of a discrete space $D$ of size continuum $A(D)=D\cup \{\infty\}$. For simplicity, in the sequel we will refer to those compacta that are premetric of degree at most $n$ for some natural number $n$, as compacta of degree $n$.

There is a nice topological characterization of those compact spaces (not necessarily Rosenthal) that are of certain degree depending on projections over countable products.
\begin{prop}\label{characterizationdegree}
Given a set $X$ and $K$ a compact subspace of the product $\mathbb{R}^X$ then $K$ is of degree $n$ if and only if there exists a countable set $D_0\subset X$ such that
$$\pi_{D_0}: K\longrightarrow \{f\upharpoonright_{D_0}: f\in K\}$$
is at most $n$-to-$1$.
\end{prop}
\begin{proof}
If there exists a $D_0$ such that $\pi_{D_0}$ is at most $n$-to-$1$, since $\pi_{D_0}$ is continuous then $\{f\upharpoonright_{D_0}: f\in K\}$ is a metric compactum and thus $K$ is a compact space of degree $n$. Conversely, let $M$ be a compact metric space and $\Phi: K\twoheadrightarrow M$ continuous and at most $n$-to-$1$.  Since $\Phi$ is a uniformly continuous function, there exists a countable subset $D_0$ of $X$ such that if
$f\upharpoonright_{D_0}=g\upharpoonright_{D_0}$ then $\Phi(f)=\Phi(g)$. This implies that $\pi_{D_0}$ is an at most $n$-to-$1$ map.
\end{proof}

\section{The compact spaces $\boldsymbol{S_n(I)}$ and $\boldsymbol{D_n(2^\mathbb{N})}$}\label{sec::section3}
Ths section will be devoted to present the compact spaces $S_n(I)$ and $D_n(2^\mathbb{N})$. Thy are the respective generalizations of the critical compacta $S(I)$ and $D(2^\mathbb{N})$ and, as we shall see on the last section, they play the same role as the Split Interval and the Alexandroff Duplicate but on the class of separable Rosenthal compacta of degree $n$. For the sake of completeness, first we show the definition of $S(I)$ and $D(K)$ as well as some of their fundamental properties.
\subsection{The compact space $S_n(I)$}
\begin{defi}[Split Interval] 
The Split Interval $S(I)$ is the space $I\times\{0,1\}\setminus\{(0,0),(1,1)\}$ endowed with the topology induced by the lexicographical order.
\end{defi}
It is easy to see that a neighbourhood basis for points $(x,0)$, $(x,1)$ in $S(I)$ is given respectively by the following sets 
\begin{eqnarray*}
](y,0),(x,1)[\;\;\text{with}\;\; y\in I,\,y<x.\\
](x,0),(y,1)[\;\;\text{with}\;\; y\in I,\,x<y.
\end{eqnarray*}
It is convenient to make some comments about the convergence on $S(I)$. Let us note that, given any sequence $\{(x_m,i_m)\}_{m\in\mathbb{N}}$ in $S(I)$, it converges to a point of the first level $(x,0)$ if and only if $\{x_m\}_{m\in\mathbb{N}}$ converges to $x$ from the left; that is, for every $\varepsilon>0$ there exists a natural number $m_0$ such that $x-\varepsilon<x_m\leq x$ for every $m\geq m_0$. The same thing occurs analogously for points of the second level and convergence from the right.

It is easy to see that $S(I)$ is a Rosenthal compactum. Indeed, the map given by $(x,0)\mapsto \mathds{1}_{[0,x)}$ and $(x,1)\mapsto \mathds{1}_{[0,x]}$ is a homeomorphism between $S(I)$ and a closed subspace of the compact space $[0,1]^{[0,1]}$ formed by functions in $\boldsymbol{\mathcal{B}_1}([0,1])$. Moreover, the Split Interval is a non-metrizable and hereditarily separable space. This second statement is immediate since $S(I)$ is separable and thus hereditarily separable by virtue of Theorem \ref{hereditarily}. The first assertion is also straightforward, in fact the only metrizable subspaces of $S(I)$ are the countable ones. We shall see later that the space $D(2^\mathbb{N})$ is not separable, so $S(I)$ does not contain copies of  $D(2^\mathbb{N})$ precisely because $S(I)$ is hereditarily separable. We summarize in the following theorem some of the most important properties of $S(I)$ that will be necessary later:
\begin{teo}[Properties of $S(I)$]\label{propertiesS(I)}
The Split Interval $S(I)$ has the following properties:
\begin{enumerate}
\item It is Rosenthal compactum of degree $2$.
\item It is hereditarily separable.
\item It is non-metrizable so it is not of degree $1$.
\item It does not contain copies of $D(2^\mathbb{N})$.
\end{enumerate}
\end{teo}
Now, we are going to present the compact space $S_n(I)$ which is a natural generalization of the Split Interval. In fact, we give the following slightly more general definition, for any perfect subset $R\subset I$:
\begin{defi}[$n-$Split Interval of $R$]
Given any natural number $n\geq 2$ and any perfect subspace $R\subset I$ the $n-$Split Interval of $R$ $S_n(R)$ is the space $R\times\{0,\ldots, n-1\}$ endowed with the topology for which the points of the form $(x,i)$ with $i\in\{2,\ldots, n-1\}$ are isolated and the points of the form $(x,0)$ and $(x,1)$ have respective basic neighborhoods of the form
\begin{eqnarray*}
\{(x,0)\} \cup \{(y,i) : z_0<y<x, i\in \{0,1,\ldots,n\}\}\\
\{(x,1)\} \cup \{(y,i) : z_1>y>x, i\in \{0,1,\ldots,n\}\}
\end{eqnarray*}
where $z_0,z_1\in R$ with $z_0<x$, $z_1>x$. If $x$ is the minimum of $R$, then $(x,0)$ is isolated, and if it is the maximum, then $(x,1)$ is isolated. If $R$ is the whole space $I$ then the we simply say that $S_n(I)$ is the $n-$Split Interval.
\end{defi}
With this definition, the $2-$Split Interval $S_2(I)$ is the Split Interval $S(I)$ with two extra isolated points. This formal difference between $S(I)$ and $S_2(I)$ is completely irrelevant for our discussion. 
Now we prove some basic results about $S_n(I)$. 
\begin{prop}\label{copyofSnI}
Given any natural number $n\geq 2$ and $R$ any perfect subspace of $I$ then $S_n(R)$ contains copies of $S_n(I)$.
\end{prop}
\begin{proof}
It is well known that every perfect Polish space contains copies of the Cantor space $2^\mathbb{N}$. Now, let us view the Cantor space $2^\mathbb{N}$ as a subset $C$ of the unit interval in the classical way. By the preceding comment, $S_n(R)$ contains copies of the space $S_n(C)$. Let us note that the subspace $C\times\{0,1\}$ of $S_n(C)$ has countably many isolated points. Namely, the set of its isolated points is the set $(\mathcal{L}\times\{1\})\cup(\mathcal{R}\times\{0\})$ where $\mathcal{L}$ and $\mathcal{R}$ are respectively the set of left and right end points deleted while building $C$. Now consider $\mathscr{X}$ the subspace of $S_n(C)$ obtained by removing the points $(\mathcal{L}\times\{1\})\cup (\mathcal{R}\times\{0,2,\ldots, n-1\})$. Since all those points are isolated in $S_n(C)$, the space $\mathscr{X}$ is compact. Then the map $$
\begin{array}{cccc}
\Psi:& \mathscr{X}& \longrightarrow& S_n(I)\\
& (x,i)& \mapsto& (\Phi(x),i)
\end{array}
$$
where $\Phi$ is the classical surjection between $C$ and $I$, is a homeomorphism. Thus, $S_n(R)$ contains copies of $S_n(I)$.
\end{proof}
\begin{prop}
$S_n(I)$ is a topological Hausdorff compact space.
\end{prop}
\begin{proof}
It is straightforward to see that $S_n(I)$ is Hausdorff so we only prove compactness. Given any open covering $\{\mathcal{U}_i\}_{i\in\mathcal{I}}$ of $S_n(I)$ we can suppose without loss of generality that $\mathcal{I} = \mathcal{I}_0\cup \mathcal{I}_1$ and $\mathcal{U}_i$ is a singleton if $i\in\mathcal{I}_0$, while if $\i\in\mathcal{I}_1$, the open set $\mathcal{U}_i$ is of the basic form 
$$\mathcal{U}_i= \left](x_i,0),(y_i,1)\right[\;\;\cup\;\;]x_i,y_i[\times\{2,\ldots, n-1\},$$
where $]a,b[$ is denoting interval in the lexicographical order of $I\times \{0,1\}$ or in the usual order of $I$. Let us note that the subspace $I\times\{0,1\}$ of $S_n(I)$ is homeomorphic to $S_2(I)$, which is compact. So, we can find a finite $F\subset \mathcal{I}_1$ such that $\{\mathcal{U}_i\}_{i\in F}$ covers $I\times\{0,1\}$. Let also $\{\mathcal{U}_i\}_{i\in G}$ be a finite subfamily of $\{\mathcal{U}_i\}_{i\in \mathcal{I}}$ which contains the points of the form $(x_i,j)$ and $(y_i,j)$ with $i\in F$ and $0\leq j\leq n-1$. We claim that the family $\{\mathcal{U}_i\}_{i\in F\cup G}$ is a finite subcover of $\{\mathcal{U}_i\}_{i\in \mathcal{I}}$. Indeed, given $(x,j)\in S_n(I)$, we have that $(x,0)\in I\times\{0,1\}$, so there exists some index $i_0\in F$ such that $(x,0)\in\mathcal{U}_{i_0}$. If $x\neq y_{i_0}$, then  $(x,j)\in\mathcal{U}_{i_0}$, by the form of $\mathcal{U}_i$. If otherwise $x = y_{i_0}$, then $(x,j)$ lies in some open set $\mathcal{U}_k$ with $k\in G$.
\end{proof}
\begin{teo}
$S_n(I)$ is a Rosenthal compactum.
\end{teo}
\begin{proof}
Let us remind that the class of Rosenthal compacta is closed under taking closed subspaces and countable products.
Now, consider the map $\Phi: S_n(I)\longrightarrow S_2(I)\times A(I)^{n-2}$ given by
\begin{eqnarray*}
\Phi(x,0)=((x,0),\infty,\ldots, \infty)\hspace{4.65cm}\\
\Phi(x,1)=((x,1),\infty,\ldots, \infty)\hspace{4.65cm}\\
\Phi(x,i)=((x,1),\infty,\ldots,\underbrace{x}_{i},\ldots, \infty)\;\; \text{if}\; i\in\{2,\ldots, n-1\}
\end{eqnarray*}
We will prove that $\Phi$ is an embedding between $S_n(I)$ and $S_2(I)\times A(I)^{n-2}$ and thus $S_n(I)$ will be a Rosenthal compactum. Since $\Phi$ is injective it shall be enough to check that is continuous. It is obvious that the projection onto the first coordinate \linebreak$\pi_1\circ \Phi: S_n(I)\longrightarrow S_2(I)$ is a continuous map. For the others, it will be enough to study the points of the form $(x,0)$ since the case $(x,1)$ is analogous and the others are isolated. Let us note that, when $i>1$,
$$
\begin{array}{cccc}
\pi_i\circ \Phi:& S_n(I)&\longrightarrow &A(I)\\
& (x,j)& \mapsto & \begin{cases}
\infty\;\;\;\text{if}\;\; j\neq i\\
x\;\;\;\text{if}\;\; j=i
\end{cases}
\end{array}
$$
Since a neighbourhood of $\pi_i\circ\Phi(x,0)=\infty$ is of the form $\mathcal{V}= I\setminus F\cup\{\infty\}$ where $F$ is a finite subset of $I$, then the neighbourhood
$$\mathcal{U}= S_n(I) \setminus\{(y,i) : y\in F, i\in\{2,\ldots,n-1\}\}$$
satisfies $\Phi(\mathcal{U})\subset \mathcal{V}$.
\end{proof}
If $n\geq 2$ then $S_n(I)$ is non-metrizable. Moreover, we now prove that the $S_n(I)$ is a Rosenthal compactum of degree $n$ but not of degree $n-1$.
\begin{teo}\label{degreeNSplit}
$S_n(I)$ is a Rosenthal compactum of degree $n$ but not of degree $n-1$.
\end{teo}
\begin{proof}
$S_n(I)$ is obviously a Rosenthal compactum of degree $n$ since the projection onto the first coordinate is continuous and at most $n$-to-$1$. To see that $S_n(I)$ is not of degree $n-1$ we need Proposition \ref{characterizationdegree}. First of all notice that the family of characteristic functions  
$$\mathscr{F}=\{\mathds{1}_{[x,y] \times \{0,\ldots,n-1\}\setminus\{(x,0),(y,1)\}}\}_{x,y\in I, x<y}\cup\{\mathds{1}_{\{(x,i)\}}\}_{x\in I,i\in\{2,\ldots,n-1\}}$$
is formed by continuous functions and separates points and thus the evaluation map $e:(x,i)\mapsto e_{(x,i)}$ establishes an embedding between $S_n(I)$ and $\mathbb{R}^\mathscr{F}$. If $S_n(I)$ were of degree $n-1$ by virtue of Proposition \ref{characterizationdegree} there would exist a countable subset $D_0$ of $\mathscr{F}$ such that $\pi_{D_0}$ is at most $(n-1)$-to-$1$ on the image of the above embedding. We are going to see that in fact given any countable subset of $\mathscr{F}$ the respective projection map is at least $n$-to-$1$. Given $N$ 
$$N=\{\mathds{1}_{[x_k,y_k] \times \{0,\ldots,n-1\}\setminus\{(x_k,0),(y_k,1)\}}\}_{k\in\mathbb{N}}\cup\{\mathds{1}_{\{(z_k,i_k)\}}\}_{k\in\mathbb{N}}$$
a countable subset of $\mathscr{F}$, we can pick a point $x\in I$ such that $x\not\in \{x_k,y_k,z_k : k\in\mathbb{N}\}$. Now, if we take any $i\in\{0,1,\ldots,n-1\}$ and any $k\in\mathbb{N}$, then
$$ \mathds{1}_{[x_k,y_k] \times \{0,\ldots,n-1\}\setminus\{(x_k,0),(y_k,1)\}}((x,i)) = \mathds{1}_{[x_k,y_k]}(x)$$
$$\mathds{1}_{(z_k,i_k)}((x,i)) = 0$$
Thus, all evaluations $\{e_{(x,i)} : i = 0,1,\ldots,n-1\}$ have the same image under $\pi_N$, so this is at least $n$-to-1.
\end{proof}
If $R$ is a perfect subspace of $I$, then $S_n(R)$ is a closed subset of $S_n(I)$, and by Proposition~\ref{copyofSnI} also $S(I)$ is homeomorphic to a closed subset of $S_n(R)$. Therefore, the previous results for $S_n(I)$ hold for $S_n(R)$ as well.
\begin{teo}[Properties of $S_n(R)$]\label{PropertiesS_n}
For every perfect subspace $R$ of the unit interval $I$, the $n-$Split Interval of $R$ has the following properties:
\begin{enumerate}
\item It is a Rosenthal compactum of degree $n$ but not of degree $n-1$.
\item It is non-separable.
\item It is non-metrizable.
\item It contains copies of the $n-$Split Interval $S_n(I)$.
\end{enumerate}
\end{teo}
\subsection{The compact space $D_n(2^\mathbb{N})$}
\begin{defi}[Alexandroff Duplicate]
Given any topological Hausdorff space $(X,\mathscr{T})$ we define its Alexandroff Duplicate $D(X)$ as the space $X\times\{0,1\}$ endowed with the topology for which the points $(x,1)$ are isolated and the points $(x,0)$ have neighbourhoods of the form
$$\mathscr{U}\times\{0,1\}\setminus\{(x,1)\}$$
where $\mathscr{U}$ is a $\mathscr{T}-$neighbourhood of $x$.
\end{defi}
It is well-kwown that if $K$ is a compact Hausdorff space then $D(K)$ is a compact Hausdorff space too. In fact, the space $D(K)$ is compact if and only if $K$ is a compact space. If $K$ is a compact metric space, then $D(K)$ is a Rosenthal compactum \cite{Tod}. In this paper we are interested in the compact space $D(2^\mathbb{N})$. It is a non separable space (because it has uncountably many isolated points) and thus is not metrizable. The results from \cite{Tod}, show that the Alexandroff Duplicate $D(2^\mathbb{N})$ is a critical example of Rosenthal compactum. For example, it embeds into any separable Rosenthal compactum of degree $2$ that is not hereditarily separable. 

\begin{prop}[Properties of $D(2^\mathbb{N})$]\label{propertiesDuplicate}
The Alexandroff Duplicate $D(2^\mathbb{N})$ has the following properties:
\begin{enumerate}
\item It is a Rosenthal compactum of degree $2$.
\item It is not separable.%
\item It is not metrisable so it is not of degree $1$.
\item It is monolithic (every separable subspace is metrizable) so it does not contain copies of $S(I)$.
\end{enumerate}
\end{prop}

We shall give a proof of the above facts for the more general spaces $D_n(2^\mathbb{N})$ that we define now: 

\begin{defi}[Alexandroff $n-$plicate]
Given any natural number $n\geq 2$ and a topological Hausdorff space $(X,\mathscr{T})$ the Alexandroff $n-$plicate $D_n(X)$ is the space \linebreak$X\times\{0,\ldots, n-1\}$ endowed with the topology for which the points $(x,i)$ with \linebreak$i\in\{1,\ldots, n-1\}$ are isolated and the points $(x,0)$ have basic neighbourhoods of the form
$$\mathscr{U}\times\{0,\ldots, n-1\}\setminus\bigcup_{i=1}^{n-1} \{(x,i)\}$$
where $\mathscr{U}$ is a $\mathscr{T}-$neighbourhood of $x$.
\end{defi}
Let us note that the Alexandroff $2-$plicate $D_2(X)$ coincides with its classical version $D(X)$. Moreover, the Alexandroff $n-$plicate share with the Alexandroff Duplicate all the properties mentioned in Proposition \ref{propertiesDuplicate}. Since we are interested in the concrete Alexandroff $n-$plicate $D_n(2^\mathbb{N})$, we only prove the fundamental properties of that space although some of these facts work for more general $D_n(X)$ spaces.
\begin{prop}
Given any natural number $n\geq 2$, the Alexandroff $n-$plicate $D_n(2^\mathbb{N})$ is a Rosenthal compactum.
\end{prop}
\begin{proof}
The space is clearly Hausdorff. The space $D_n(X)$ is compact whenever $X$ is compact, because if $\mathcal{W}$ is an open cover, after we take a finite subcover of $X\times \{0\}$, only finitely many points can remain, or otherwise they would accumulate to $X\times\{0\}$. To see that it is a Rosenthal compactum, let us consider the map $\Phi: D_n(2^\mathbb{N})\longrightarrow 2^\mathbb{N}\times A(2^\mathbb{N})^{n-1}$ given by $$\Phi(x,0)=(x,\infty,\ldots, \infty)$$ $$\Phi(x,i)=(x,\infty,\ldots,\underbrace{x}_{i}, \ldots, \infty).$$ Since $2^\mathbb{N}\times A(2^\mathbb{N})^{n-1}$ is a Rosenthal compactum it will be enough to see that $\Phi$ is an embedding and, in fact, it will be enough to see that $\Phi$ is continuous on the points $(x,0)$. A basic neighborhood of $\Phi(x,0)$ is of the form
$$\mathscr{U}\times\prod_{i=1}^{n-1}\left( 2^\mathbb{N}\setminus F_i\cup\{\infty\}\right)$$
where $\mathscr{U}$ is a neighbourhood of $x$ in $2^\mathbb{N}$ and $F_i$ are finite subsets of $2^\mathbb{N}$. Let $$F = \bigcup_{i=1}^{n-1}\{(x,i)\}\cup \{(y,j) : i,j\in\{1,\ldots, n-1\}, y\in F_i\}$$ and consider
$\mathscr{V} = \mathscr{U}\times\{0,\ldots, n-1\}\setminus F$ which is a neighborhood of $(x,0)$ that satisfies
$$\Phi(\mathscr{V})\subseteq \mathscr{U}\times\prod_{i=1}^{n-1}\left(\{y:y\in K\setminus F_i\}\cup\{\infty\}\right)$$
and thus $\Phi$ is a continuous map.
\end{proof}
\begin{prop}\label{degreeofDn}
Given any natural number $n\geq 2$ the Alexandroff $n-$plicate $D_n(2^\mathbb{N})$ is a Rosenthal compactum of degree $n$ that is not of degree $n-1$.
\end{prop}
\begin{proof}
It is obvious that $D_n(2^\mathbb{N})$ is a Rosenthal compactum of degree $n$ since the corresponding projection onto $2^\mathbb{N}$ is a $n$-to-1 continuous map. To see that $D_n(2^\mathbb{N})$ is not of degree $n-1$ we proceed in the same way as in Theorem \ref{degreeNSplit}. Since $2^\mathbb{N}$ is zero dimensional and second-countable space, we can take a countable clopen basis $\{\mathscr{A}_m\}_{m\in\mathbb{N}}$ and thus if $\mathscr{F}=\{\mathds{1}_{\mathcal{A}_m\times\{0,\ldots, n-1\}}\}_{m\in\mathbb{N}}\cup \{\mathds{1}_{\{(x,i)\}}\}_{x\in 2^\mathbb{N},i\in\{1,\ldots, n-1\}}$ then the evaluation map $e:(x,i)\mapsto e_{(x,i)}$ establish an embedding between $D_n(2^\mathbb{N})$ and $\mathbb{R}^\mathscr{F}$.

Suppose that $D_n(2^\mathbb{N})$ is of degree $n-1$. Then by Proposition \ref{characterizationdegree} there exists a countable subset $N$ of $\mathscr{F}$ for which the corresponding projection $\pi_N$ is at most $(n-1)$-to-$1$. We can suppose that $N$ is of the form $N=\{\mathds{1}_{\mathcal{A}_{m_k}\times\{0,\ldots, n-1\}}\}_{k\in\mathbb{N}}\cup \{\mathds{1}_{\{(x_k,i_k)\}}\}_{k\in\mathbb{N}}$. We can pick a point $x\in 2^\mathbb{N}$ for which $x\neq x_k$ for every $k\in\mathbb{N}$. Then we have that, for all $k\in\mathbb{N}$ and all $i\in\{0,\ldots,n-1\}$,
$$e_{(x,i)}(\mathds{1}_{\{(x_k,i_k)\}})=0,$$
$$e_{(x,i)}(\mathds{1}_{\mathcal{A}_{m_k}\times\{0,\ldots, n-1\}})= \mathds{1}_{\mathcal{A}_{m_k}}(x).$$
Therefore $\pi_N$ is at least $n$-to-$1$ because all points $(x,i)$ have the same image. A contradiction.
\end{proof}
We finish the present section proving that the Alexandroff $n-$plicate is a monolithic space. As we shall see later, this will imply that none of the compact spaces $S_n(I)$ and $D_n(2^\mathbb{N})$ contain copies of the other. 
\begin{prop}
Given any natural number $n\geq 2$ the Alexandroff $n-$plicate $D_n(2^\mathbb{N})$ is a monolithic space.
\end{prop}
\begin{proof}
Let $Z$ be a countable subset of $D_n(2^\mathbb{N})$, and let $\mathcal{B}$ be a countable basis of open subsets of $2^{\mathbb{N}}$. Then,
$$\{ B\times \{0,1,\ldots,n-1\} : B\in \mathcal{B}\} \cup \{\{(x,i)\} : (x,i)\in Z, i\neq 0\}$$
is a countable family of open sets that separates the points of the compact space $\overline{Z}$ (notice that all points of $\overline{Z}\setminus Z$ live in $2^\mathbb{N}\times \{0\}$). Hence $\overline{Z}$ is metrizable. 
\end{proof}
\begin{cor}
Given any natural number $n\geq 2$ then $S_n(I)$ does not contain copies of $D_n(2^\mathbb{N})$ and vice versa.
\end{cor}
\begin{proof}
If $S_n(I)\hookrightarrow D_n(2^\mathbb{N})$, since the Alexandroff $n-$plicate is a monolithic space and $S(I)$ is a separable subspace of $S_n(I)$ (remember that $S(I)$ is the result of removing two isolated points from $S_2(I)$), then $S(I)$ would be metrizable contradicting Proposition \ref{propertiesS(I)}. On the other hand, if $D_n(2^\mathbb{N})\hookrightarrow S_n(I)$, then the non-isolated points of $D_n(2^\mathbb{N})$ embed inside the non-isolated points of $S_n(I)$. That would give a copy of $2^\mathbb{N}$ inside $S(I)$ in contradiction, again, with Proposition \ref{propertiesS(I)}.
\end{proof}
\section{The proof of the main result}\label{sec::section4}

The present section will be devoted to the proof of our main result. First of all, it is convenient to point out that the assumption of separability is essential. Indeed, let us consider the Rosenthal compact space $K=D_n(2^\mathbb{N})\setminus \mathscr{B}$ where $\mathscr{B}=B\times\{1,\ldots,n-1\}$ and $B$ is a Bernstein set (that is, a set such that $C\cap B$ and $C\setminus B$ are uncountable for all uncountable Borel sets). The same proof as in Proposition~\ref{degreeofDn} shows that $K$ is of degree $n$ but not $n-1$. It is clear that $K$ does not contain copies of $S_n(I)$, not even $S_2(I)$, since neither does $D_n(2^\mathbb{N})$. On the other hand, $K$ neither contains copies of even $D_n(2^\mathbb{N})$, not even $D_2(2^\mathbb{N})$. One way to see this is to use the fact that for every continuous function $f:D_2(2^\mathbb{N})\longrightarrow 2^\mathbb{N}$ there exists a perfect subset $P\subset 2^\mathbb{N}$ such that $|f^{-1}(p)|=2$ for all $p\in P$. The restriction of the first coordinate $f:X\longrightarrow 2^\mathbb{N}$, $f(x,i) = x$ has points with only one preimage on any perfect set, so we conclude that $D_2(2^\mathbb{N})$ does not embed into $X$. The proof of the above fact is an elementary excercise: Por every clopen set $A\subset 2^\mathbb{N}$, write its clopen preimage in the form 
$$f^{-1}(A) = C_A \times \{0,1\}\bigtriangleup F_A \times\{1\},$$ where $C_A\subset 2^\mathbb{N}$ is clopen, $F_A$ is finite, and $X \bigtriangleup Y$ represents the symmetric difference $(X\setminus Y)\cup (Y\setminus X)$. It is enough to take $P$ a perfect set disjoint from all the $F_A$'s.

As in the proof of the classical result for separable Rosenthal compacta of degree $2$, we will require the Ramsey-like results used in \cite{Tod}. For the sake of completeness, we state both in the following lines.
\begin{teo}[S. Todorcevic]\label{Ramsey}
Let $\left\{f_s : s\in 2^{<\mathbb{N}}\right\}$ be a relatively compact subset of Baire class 1 functions defined on a Polish space $X$. Then there is a perfect Polish space $P\subset 2^\mathbb{N}$ and an infinite strictly increasing sequence $\{m_k\}_k$ of natural numbers such that $\{f_{a\upharpoonright m_k}\}_k$ is pointwise convergent for every $a\in P$.
\end{teo}
\begin{teo}[F. Galvin]\label{Ramsey2}
For every perfect Polish space $X$ and every symmetric analytic relation $A\subset X^2$ there is a perfect set $P\subset X$ such that $P^{[2]}=\{(x,y)\in P^2:\; x\neq y\}$ is either disjoint from or included in $A$.
\end{teo}
\subsection{The proof}
In the sequel, $n\geq 2$ will be a natural number, $K\subset \boldsymbol{\mathcal{B}_1}(\mathbb{N}^\mathbb{N})$ will be a separable Rosenthal compactum of degree $n$ but not of degree $n-1$ and $K_0$ will be a countable dense subspace of $K$ made of continuous functions (remember Proposition~\ref{separableK}). 

Since $K$ is of degree $n$, by virtue of Proposition \ref{characterizationdegree}, there exists a countable subset $D_0\subset \mathbb{N}^\mathbb{N}$ such that $\pi_{D_0}$ is at most $n$-to-$1$. With this in mind we will build, as it was done in \cite{Tod}, an $\omega_1-$sequence of functions $\{(f_\alpha^0,\ldots, f^{n-1}_\alpha)\}_{\alpha<\omega_1}$ on $K^n$ and an $\omega_1$-sequence of points $\{(x^{ij}_\alpha)_{ij}:\,0\leq i\neq j\leq n-1,\,\alpha<\omega_1\}$ on $(\mathbb{N}^\mathbb{N})^{n^2-n}$ with the following properties:
\begin{description}
\item[(1)]$f^i_\alpha\upharpoonright_{D_0}=f^j_\alpha\upharpoonright_{D_0}$ for every $i,j\in\{0,\ldots, n-1\}$.
\item[(2)]$f^i_\alpha(x^{ij}_\alpha)\neq f^{j}_\alpha(x^{ij}_\alpha)$ and $x^{ij}_\alpha=x^{ji}_\alpha$ for all $\alpha<\omega_1$ and $i,j\in\{0,\ldots, n-1\}$ with $i\neq j$.
\item[(3)]$f^k_\alpha(x^{ij}_\beta)=f^l_\alpha(x^{ij}_\beta)$ for all $\beta<\alpha$ and $i,j,k,l\in\{0,\ldots, n-1\}$ with $i\neq j$.
\item[(4)]$f^i_\alpha\upharpoonright_{D_0}\neq f^j_\beta\upharpoonright_{D_0}$ when $\alpha\neq \beta$ and $i,j\in\{0,\ldots, n-1\}$.
\end{description}
\begin{prop}
There exist two $\omega_1-$sequences $$\{(f_\alpha^0,\ldots, f^{n-1}_\alpha)\}_{\alpha<\omega_1},$$ $$\{(x^{01}_\alpha, \ldots, x^{(n-1)(n-2)}_\alpha)\}_{\alpha<\omega_1},$$ as above, which satisfy \textbf{(1)-(4)}.
\end{prop}
\begin{proof}
We proceed by induction. Given any $\alpha<\omega_1$ let us consider $$D_\alpha=D_0\cup\{(x^{01}_\gamma, \ldots, x^{(n-1)(n-2)}_\gamma)\}_{\gamma<\alpha}$$ and suppose that for every $\gamma<\alpha$ the tuples $(f^0_\gamma,\ldots,f^{n-1}_\gamma)$ and $(x^{01}_\gamma, \ldots, x^{(n-1)(n-2)}_\gamma)$ satisfy \textbf{(1)-(4)}. Since $K$ is not of degree $n-1$, by Proposition \ref{characterizationdegree}, $\pi_{D_\alpha}$ is at least $n$-to-$1$. Therefore there exist $f^0_\alpha,\ldots, f^{n-1}_\alpha$ in $K$ satisfying 
$$f^i_\alpha\upharpoonright D_\alpha =f^j_\alpha\upharpoonright D_\alpha,\;\;\; i,j\in\{0,\ldots n-1\}.$$
In particular \textbf{(1)} and \textbf{(3)} are satisfied.

Since all the functions are distinct, there exist $n^2-n$ points $x^{ij}_\alpha$ ($x^{ij}_\alpha=x^{ji}_\alpha$ for every $i\neq j$) distinguishing $f^i_\alpha$ and $f^j_\alpha$, in the sense that \textbf{(2)} is satisfied. It remains to see that \textbf{(4)} is true. By contradiction, let us suppose that there exists $\gamma<\alpha$ and $i_0,j_0\in\{0,\ldots, n-1\}$ such that $f^{i_0}_\alpha\upharpoonright_{D_0}=f^{j_0}_\gamma\upharpoonright_{D_0}$. Then by construction $f^{i}_\alpha\upharpoonright_{D_0}=f^{j}_\gamma\upharpoonright_{D_0}$ for every $i,j$. Since $\pi_{D_0}$ is at most $n$-to-$1$, there are indices $k_0$ and $k_1$ such that $f^{k_0}_\alpha=f^0_\gamma$ and $f^{k_1}_\alpha=f^1_\gamma$. But this is impossible because $f^{k_0}_\alpha(x_\gamma^{01})=f^{k_1}_\alpha(x_\gamma^{01})$ by \textbf{(3)} whereas $f^0_\gamma(x_\gamma^{01})\neq f^1_\gamma(x_\gamma^{01})$ by property \textbf{(2)}.
\end{proof}
Due to property \textbf{(2)}, passing to an uncontable subset of $\omega_1$, we can assume that there exist open intervals $I^{ij} \subset \mathbb{R}$ with rational endpoints such that:
\begin{description}
\item[(5)] For every $\alpha<\omega_1$ and $i,j\in\{0,\ldots, n-1\}$ different indices, we have
\begin{eqnarray*}
\overline{I^{ij}}\cap \overline{I^{ji}}=\emptyset\\
f^i_\alpha(x^{ij}_\alpha)\in I^{ij}\\
f^j_\alpha(x^{ij}_\alpha)\in I^{ji}
\end{eqnarray*}
\end{description}
Also we can suppose without loss of generality that either $I^{ij}=I^{kl}$ or $I^{ij}\cap I^{kl}=\emptyset$ for any choice of indices.
\\[0.1cm]
As in \cite{Tod}, for every $m\in\mathbb{N}$ fix an enumeration $\{B_{mk}\}_{k\in\mathbb{N}}$ of all open rational intervals with diameter $\leq 2^{-m}$ and $\{d_m\}_{m\in\mathbb{N}}$ an enumeration of $D_0$ in which every element is repeated infinitely many times. Given a finite sequence $t\in\mathbb{N}^{<\mathbb{N}}$ define 
$$B(t)=\{h\in\mathbb{R}^{D_0}:\, h(d_m)\in B_{m\,t(m)},\;\;m\leq\,length(t)\}.$$
Note that $\{B(t)\}_{t\in\mathbb{N}^{<\mathbb{N}}}$ is an open basis of $\mathbb{R}^{D_0}$. Note also that $B(s)\subset B(t)$ if $t\preceq s$.

A key point in the proof is to find an embedding between the Cantor set $2^\mathbb{N}$ and some perfect subspaces of $(\mathbb{N}^\mathbb{N})^{(n^2-n)}$ and $\mathbb{R}^{D_0}$ respectively.  

Given an ordinal $\alpha<\omega_1$ we will denote by $h_\alpha$ the common restriction to $D_0$ of the $f^i_\alpha$ functions. We will construct a tree $T$ on $(\mathbb{N}^{<\mathbb{N}})^{n^2-n}\times \mathbb{N}^{<\mathbb{N}}$ formed by pairs $(t^0,t^1)$ where $t^0$ will be a tuple that we write in the form $(t^{0}_{ij}, i,j=0,\ldots,n-1, i\neq j)$, and will be a restriction of some (in fact many) $(x^{ij}_\alpha : i,j=0,\ldots,n-1, i\neq j)$, while the second coordinate $t^1$ codifies the open sets $B(t^1)$ where the function $h_\alpha$ lies.

\begin{defi}
Given $(t^0, t^1)$ in $(\mathbb{N}^{<\mathbb{N}})^{(n^2-n)}\times \mathbb{N}^{<\mathbb{N}}$ and given $\alpha<\omega_1$, we say that the pair $(x_\alpha, h_\alpha)$ in $(\mathbb{N}^\mathbb{N})^{(n^2-n)}\times \mathbb{R}^{D_0}$ extends $(t^0, t^1)$ if $h_\alpha\in B(t^1)$ and $t^0_{ij}\subset x^{ij}_\alpha$ for all $i,j=0,\ldots,n-1$, $i\neq j$.
\end{defi}

All the components of the pairs $(t^0,t^1)$ of our tree will have same length, so from now on, we assume that for all tuples $(t^0,t^1)\in (\mathbb{N}^{<\mathbb{N}})^{(n^2-n)}\times \mathbb{N}^{<\mathbb{N}}$, we have that $length(t^0_{ij}) = length(t^1)$ for all $i,j$. In the sequel, the expression $\exists^{\omega_1} \alpha$ means \emph{there exist $\omega_1$ many $\alpha$'s such that...}.

\begin{defi}
Given a node $(t^0,t^1)\in (\mathbb{N}^{<\mathbb{N}})^{(n^2-n)}\times \mathbb{N}^{<\mathbb{N}}$ we say that two extensions $(s^0,s^1)$ and $(u^0,u^1)$ of $(t^0,t^1)$ properly split if 
$$s^0\neq u^0\;\;\wedge\;\; \exists m\in\mathbb{N}\,\left(\,s^1(m)\neq t^1(m)\;\wedge\; \overline{B_{ms^1(m)}}\cap\overline{B_{mt^1(m)}}=\emptyset\,\right).$$
\end{defi}

\begin{prop}
There exists a tree $T$ on the product $(\mathbb{N}^{<\mathbb{N}})^{(n^2-n)}\times \mathbb{N}^{<\mathbb{N}}$ formed by pairs $(t^0,t^1)$ for which the following condition holds
$$\exists^{\omega_1}\,\alpha\;\;\;\text{such that}\;\;\; (x_\alpha,h_\alpha)\;\;\text{extends}\;\; (t^0,t^1).$$
Moreover, for every $(t^0,t^1)$ in $T$ there exist extensions $(s^0,s^1)$ and $(u^0,u^1)$ of $(t^0,t^1)$ in $T$ that properly split.
\end{prop}
\begin{proof}
First of all define
$$\Upsilon=\{(t^0,t^1)\in(\mathbb{N}^{<\mathbb{N}})^{(n^2-n)}\times \mathbb{N}^{<\mathbb{N}}:\exists^{\omega_1}\alpha\;\text{such that}\; (x_\alpha,h_\alpha)\;\text{extends}\; (t^0,t^1)\}$$
We will prove that $\Upsilon$ has a subtree $T$ which satisfies the statement of the proposition. Let us define the following subset of $\omega_1$
\begin{equation}\label{A}
\mathscr{A}=\{\alpha<\omega_1:\, \text{if}\;(x_\alpha,h_\alpha)\;\text{extends}\; (t^0,t^1) \;\text{then}\; (t^0,t^1)\in\Upsilon\}.
\end{equation}
Given any $(t^0,t^1)$ denote by $\mathscr{E}_{(t^0,t^1)}$ the set of all ordinals $\alpha$ such that $(x_\alpha, h_\alpha)$ extends $(t^0,t^1)$. Notice that $\mathscr{A}$ is uncountable. In fact, $\omega_1\setminus\mathscr{A}$ is countable because
$$\omega_1\setminus\mathscr{A}\subseteq\bigcup_{(t^0,t^1)\notin\Upsilon} \mathscr{E}_{(t^0,t^1)}$$
and $\mathscr{E}_{(t^0,t^1)}$ is countable whenever $(t^0,t^1)\not\in \Upsilon$. With this in mind define
$$T=\{(t^0,t^1)\in (\mathbb{N}^{<\mathbb{N}})^{(n^2-n)}\times \mathbb{N}^{<\mathbb{N}} : \exists\,\alpha\in \mathscr{A}\;\text{such that}\; (x_\alpha, h_\alpha)\; \text{extends}\; (t^0,t^1)\}.$$
Since we said that $\omega_1\setminus \mathscr{A}$ is countable and $T\subset \Upsilon$, it is clear that given any $(t^0,t^1)\in T$ the set of all ordinals in $\alpha\in \mathscr{A}$ such that $(x_\alpha,h_\alpha)$ extends $(t^0,t^1)$ is uncountable. Let us see that $T$ satisfies the statement of the proposition. Given a node $(t^0,t^1)\in T$ let us consider two different ordinals $\alpha,\beta\in\mathscr{A}$ such that $(x_\alpha, h_\alpha)$ and $(x_\beta, h_\beta)$ extend $(t^0,t^1)$. Since $x_\alpha$ and $x_\beta$ are different, we can pick some natural number $m$ such that $x_\alpha\upharpoonright m$ and $x_\beta\upharpoonright m$ are different. For the second coordinate, we can find some $d\in D_0$ and $l\geq m$ such that
\begin{equation}\label{desi}
|h_\alpha(d)-h_\beta(d)|>\frac{1}{2^l}
\end{equation}
Since in $D_0$ every element is repeated infinitely many times, we can also pick an index $k>l$ such that $d_k=d$. We can find sequences $t',\bar{t}\in \mathbb{N}^{k+1}$ such that $(x_\alpha\upharpoonright k+1, t'), (x_\beta\upharpoonright k+1, \overline{t})$ extend $(t^0,t^1)$ and lie in $T$. Moreover $h_\alpha(d_k)\in B_{kt'(k)}$ and $h_\beta(d_k)\in B_{k\overline{t}(k)}$. Due to the inequality (\ref{desi}) above we conclude that $B_{kt'(k)}$ and $B_{k\overline{t}(k)}$ have disjoint closures and then $(x_\alpha\upharpoonright k, t'), (x_\beta\upharpoonright k, \overline{t})$ are the extensions of $(t^0,t^1)$ that we were looking for.
\end{proof}

In the same way as in \cite{Tod} we build, by induction on $\sigma\in 2^{<\mathbb{N}}$, nodes $t_\sigma = (t^0_\sigma, t^1_\sigma)\in T$ and $n-$tuples $(g^0_\sigma,\ldots, g^{n-1}_\sigma)\in K^n_0$ which are respective approximations to $(x_\alpha,h_\alpha)$ and $(f^0_\alpha,\ldots, f^{n-1}_\alpha)$ and satisfy the following conditions:
\begin{description}
\item[(6)] If $\tau\in 2^{<\mathbb{N}}$ extends $\sigma$ then $(t^0_\tau, t^1_\tau)$ is an extension of $(t^0_\sigma, t^1_\sigma)$; that is, $t^0_{(ij)\sigma}\preceq t^0_{(ij)\tau}$ for every $i,j\in\{0,\ldots, n-1\}$, $i\neq j$ and $t^1_\sigma\preceq t^1_\tau$.
\item[(7)]Given $\sigma\in 2^{<\mathbb{N}}$ the sequences $(t^0_{\sigma 0}, t^1_{\sigma 0})$ and $(t^0_{\sigma 1}, t^1_{\sigma 1})$ are extensions of $(t^0_{\sigma}, t^1_{\sigma})$ that properly split.
\item[(8)]For every sequence $\sigma\in 2^{<\mathbb{N}}$ and indices $i,j\in\{0,\ldots,n-1\}$ with $i\neq j$ we have
\begin{eqnarray*}
g^i_\sigma[t^0_{(ij)\sigma}]\subset I^{ij}\\
g^j_\sigma[t^0_{(ij)\sigma}]\subset I^{ji}
\end{eqnarray*}
where remember that, for $r\in\mathbb{N}^{<\mathbb{N}}$, $[r] = \{x\in \mathbb{N}^{\mathbb{N}} : r\subset x\}$. 
\item[(9)]$g^i_\sigma\upharpoonright D_0\in B(t^1_\sigma)$ for every $\sigma\in 2^{<\mathbb{N}}$ and $i\in\{0,\ldots, n-1\}$.
\end{description}
\begin{description}
\item[(10)]If $\tau\neq \sigma$ and $\tau,\sigma\in 2^m$ then for every $i,j,k,l\in\{0,\ldots,n-1\}$ with $i\neq j$
$$\sup_{x\in [t^0_{(ij)\sigma}]} |g^k_\tau(x)-g^l_\tau(x)|\leq \frac{1}{2^m}$$.
\end{description}

This is done by induction, let us explain how. For the purpose of this construction, if $\sigma,\tau\in 2^{<\mathbb{N}}$, $\sigma<\tau$ means that both sequences have the same length and $\sigma$ is below $\tau$ lexicographically. Along the induction, we also build auxiliary tuples $v(\tau,\xi)  \in T$ for $\tau,\xi\in 2^{<\mathbb{N}}$ of the same length $m$  such that 
\begin{itemize}
\item $t_{\tau|m-1} \preceq v(\tau,\xi) \preceq v(\tau,\xi')\preceq t_{\tau}$ whenever $\xi<\xi'$, 
\item $v(\tau 0,\xi) = v(\tau 1,\xi)$ if $\xi < \tau 0$. 
\item whenever $i\neq j$, and $\tau<\tau'$ we have
$$\sup_{x\in [v^0_{(ij)}(\tau',\xi)]} |g^k_\tau(x)-g^l_\tau(x)|\leq \frac{1}{2^{m+1}}$$.
\end{itemize}
These $v(\tau,\xi)$ are provisional values of $t_\tau$, so that the definite value is $t_\tau = v(\tau,(111\ldots 1))$.

We fix $\sigma\in 2^{<\mathbb{N}}$ of length $m$ and we suppose that all $g_\tau^i$, $t_\tau$ have been defined when $length(\tau)\leq m$ and when  $\tau<\sigma0$, and the $v(\tau,\xi)$ have been defined for all $\tau \in 2^{m+1}$ and $\xi<\sigma 0$. We shall show how to define all those objects for $\tau = \sigma 0, \sigma 1$ and for any $\xi$ of length $m+1$. Let $\sigma^{-}$ be the immediate lexicographical predecessor of $\sigma$. For notational simplicity it is convenient to consider an imaginary $0^- = (0\ldots 0)^-\in \mathbb{N}^{m+1}$ such that $v(\tau,0^-) = t_{\tau|m}$. The first thing is to find two nodes $u^0,u^1\in T$ above $v(\sigma 0, \sigma^-) = v(\sigma 1,\sigma^-)$ that properly split. Let us deal first with defining the objects associated to $\sigma 0$. For every $\tau\in 2^{m+1}$ choose $\alpha(\tau) \in \mathcal{A}$ shuch that $(x_{\alpha(\tau)},h_{\alpha(\tau)})$ extends $v(\tau,\sigma^-)$ for all $\tau$, and moreover, $(x_{\alpha(\sigma \varepsilon)},h_{\alpha(\sigma \varepsilon)})$ extends $u^\varepsilon$. We can suppose that $\alpha(\sigma 0) > \alpha(\tau)$ for all other $\tau$. Then, we can apply \textbf{(3)} and \textbf{(5)} for $\alpha=\alpha(\sigma 0)$ and $\beta = \alpha(\tau)$. By the density of $K_0$, we can find $g^i_{\sigma 0} \in K_0$ for $i=0,\ldots,n-1$ such that
$$g_{\sigma 0}^k(x_{\alpha(\tau)}) - g_{\sigma 0}^l (x_{\alpha(\tau)}) < 2^{-m-2}, \tau\neq \sigma 0$$
$$g^i_{\sigma 0} (x_{\alpha(\sigma 0)} ^{ij}) \in I^{ij}.$$
Using the continuity of $g^i_{\sigma 0}$ at the points $x_{\alpha(\tau)}^{ij}$, we can find the nodes $v(\tau,\sigma 0)$ large enough to satisfying all requirements. The case of $\sigma 1$ is done exactly in the same way. This finishes the construction.

Given $a\in 2^\mathbb{N}$ and different $i,j<n$ let us consider $x^{ij}_a\in\mathbb{N}^\mathbb{N}$ the unique extension of the sequence $\{t^0_{(ij)a\upharpoonright m}\}_m$, and also let $h_a\in \mathbb{R}^{D_0}$ be the unique function such that $h_a(d_m) \in \overline{B_{m t^1_{a\upharpoonright k}(m)}}$ for all $m$ and sufficiently large $k$. In order to make sure that such a function exists one should check that we do not get incompatible conditions when $d_m = d_{m'}$, but this is clear since every $(t^0,t^1)$ is extended by some $(x_\alpha,h_\alpha)$.

\begin{prop}
The map $a\mapsto (x^{01}_a,\ldots, x^{(n-1)(n-2)}_a)$ is a homeomorphism between $2^\mathbb{N}$ and the perfect subspace $\{(x^{01}_a,\ldots, x^{(n-1)(n-2)}_a)\in{(\mathbb{N}^\mathbb{N})}^{n^2-n}:a\in 2^\mathbb{N}\}$ of ${(\mathbb{N}^\mathbb{N})}^{n^2-n}$.
Similarly, the map $a\mapsto h_a$ is a homeomorphism between $2^\mathbb{N}$ and the perfect subspace $\{h_a:\,a\in 2^\mathbb{N}\}$ of $\mathbb{R}^{D_0}$.
\end{prop}
\begin{proof}
It is easy to see that the family $\{\mathcal{A}_{a\upharpoonright m}\}_{m\in\mathbb{N}, a\in 2^\mathbb{N}}$ where
$$\mathcal{A}_{a\upharpoonright m}=\{y\in{(\mathbb{N}^\mathbb{N})}^{n^2-n}:\;t^0_{a\upharpoonright m}\subset^* y\}$$
is a Cantor scheme. Moreover the sets $\mathcal{A}_{a\upharpoonright m}$ are clopen and for every $a\in 2^\mathbb{N}$ we have $\lim_m diam(\mathcal{A}_{a\upharpoonright m})=0$. Therefore the map $a\mapsto  (x^{01}_a,\ldots, x^{(n-1)(n-2)}_a)$ is a homeomorphism.

On the other hand, it is clear that the map $a\mapsto h_a$ is continuous because if $a\upharpoonright m+1 = b\upharpoonright m+1$ then $|h_a(d_m)-h_b(d_m)|\leq 2^{-m}$, and each $d\in D_0$ is repeated infinitely many times in the sequence $\{d_m\}$. So if we managed to prove that it is injective we would conclude that it is in fact a homeomorphism. So take $a$ and $b$ two different sequences of the Cantor space and consider $m_0$ the minimum coordinate where $a(m_0)\neq b(m_0)$. By virtue of \textbf{(7)} $(t^0_{a\upharpoonright m_0+1},t^1_{a\upharpoonright m_0+1})$ and $(t^0_{b\upharpoonright m_0+1},t^1_{b\upharpoonright m_0+1})$ properly split, so there exists $m$ such that 
$$\overline{B_{m t^1_{a\upharpoonright m_0+1}(m)}} \cap \overline{B_{m t^1_{b\upharpoonright m_0+1}(m)}} = \emptyset $$
So, in particular, $h_a(d_m) \neq h_b(d_m)$.
\end{proof}
Let us note that the family of functions $\{g^i_\sigma\}_{\sigma\in 2^{<\mathbb{N}}}$ with $i\in\{0,\ldots, n-1\}$ is a relatively compact subset of $\boldsymbol{\mathcal{B}_1}(\mathbb{N}^\mathbb{N})$. By the Ramsey-like Theorem \ref{Ramsey} we can ensure a uniform behaviour of the functions $g^i_{a\upharpoonright m}$ for many $a\in 2^\mathbb{N}$ . Namely, there exists a perfect Polish space $P\subset 2^\mathbb{N}$ and a infinite strictly increasing sequence $\{m_k\}_k$ of natural numbers for which there exists the limit $\lim_m g^i_{a\upharpoonright m} =g^i_a$ for every $a\in P$ and $i\in\{0,\ldots, n-1\}$.
\begin{prop}
The family of functions $\{(g^0_a,\ldots, g^{n-1}_a)\}_{a\in P}\subset K^n$ satisfies the following properties:
\begin{description}
\item[(11)] For every different indices $i,j\in\{0,\ldots,n-1\}$ and $a\in P$ we have
\begin{eqnarray*}
g^i_a(x^{ij}_a)\in\overline{I^{ij}}\\
g^j_a(x^{ij}_a)\in\overline{I^{ji}}
\end{eqnarray*}
\item[(12)] $g^i_a\upharpoonright D_0 = h_a$ for every $i\in\{0,\ldots, n-1\}$ and $a\in P$. In fact, $\pi^{-1}_{D_0}(h_a)=\{g^0_a,\ldots, g^{n-1}_a\}$.

\item[(13)] For every indices $i,j,k,l\in\{0,\ldots, n-1\}$ with $i\neq j$ and distinct points $a,b\in P$ we have 
$$g^k_a(x^{ij}_b)=g^l_a(x^{ij}_b)$$ 
\end{description}
\end{prop}
\begin{proof}
The properties \textbf{(11)} and \textbf{(13)} are straightforward consequences of \textbf{(8)} and \textbf{(10)}, so we only need to prove property \textbf{(12)}. If \textbf{(12)} does not hold, by the definition of $h_a$, we could find $m_1<m_0$  such that
$$g^i_a(d_{m_1})\notin \overline{B_{m_1\,t^1_{a\upharpoonright m_0}(m_1)}}.$$
 Due to the convergence of $\{g^i_{a\upharpoonright m}\}_m$, there exists a natural number $m>m_0$ such that
$$g^i_{a\upharpoonright m}(d_{m_1})\not\in \overline{B_{m_1\,t^1_{a\upharpoonright m_0}(m_1)}}.$$
But on the other hand, by \textbf{(9)}, we have $g^i_{a\upharpoonright m}\in B(t^1_{a\upharpoonright m})\subseteq B(t^1_{a\upharpoonright m_0})$, which leads to a contradiction. The last statement of \textbf{(12)} follows from the facts that $g^i_a$ and $g^j_a$ are distinct for $i\neq j$ (see \textbf{(11)} and \textbf{(5)}) and $\pi_{D_0}$ is at most $n$-to-$1$.
\end{proof}
 Now as in \cite{Tod}, choose open intervals $J^{ij}$ with rational endpoints such that:
\begin{description}
\item[(14)] $\overline{I^{ij}}\subseteq J^{ij}$ for every different indices $i,j\in\{0,\ldots, n-1\}$.
\item[(15)] $J^{ij}\cap J^{ji}=\emptyset$ for every different indices $i,j\in\{0,\ldots, n-1\}$.
\end{description}
and define the following relations on $P^2$, where $\min$ and $\max$ refer to the natural lexicographical order of $2^\mathbb{N}$:
\begin{description}
\item[(16)] $(x,y)\in A^{ij}_0$ if and only if $x\neq y$ and $g^i_a(x^{ij}_b)\in J^{ij}$ where $a=\min\{x,y\}$ and \linebreak$b=\max\{x,y\}$.
\item[(17)] $(x,y)\in A^{ij}_1$ if and only if $x\neq y$ and $g^i_a(x^{ij}_b)\in J^{ij}$ where $a=\max\{x,y\}$ and \linebreak$b=\min\{x,y\}$.
\end{description}
It is easy to see that the previous relations are symmetric and Borel so Theorem \ref{Ramsey2} applies. Iterating this procedure, we obtain a perfect subset $R\subset P$ such that given two different indices $i,j\in\{0,\ldots, n-1\}$ and $\varepsilon\in\{0,1\}$ the following holds:
\begin{description}
\item[(18)] Either $R^{[2]}\subseteq A^{ij}_\varepsilon\;\; or\;\; R^{[2]}\cap A^{ij}_\varepsilon = \emptyset.$
\end{description} 
where recall that $R^{[2]}$ denotes the set $R^{[2]}=\{(a,b)\in R^2:\, a\neq b\}$.

Given $a\in P$ and two different indices $i,j\in\{0,\ldots, n-1\}$ let us consider the following pointwise neighborhoods
\begin{eqnarray*}
\mathscr{U}_j(g^i_a)=\{f\in K:\, f(x^{ij}_a)\in J^{ij}\}.
\end{eqnarray*}
Note that the previous sets satisfy the following properties, as a direct consequence of \textbf{(15)} and \textbf{(13)}:
\begin{description}
\item[(19)] $\mathscr{U}_j(g^i_a)\cap\mathscr{U}_i(g^j_a)=\emptyset$ for every $a\in R$ and different indices $i,j\in\{0,\ldots, n-1\}$.
\item[(20)] If $a,b\in R$ and $i,j\in\{0,\ldots, n-1\}$ are different points and indices respectively, then either $g^k_b\in\mathscr{U}_j(g^i_a)$ for every $k\in\{0,\ldots, n-1\}$ or $g^k_b\notin \mathscr{U}_j(g^i_a)$ for every $k\in\{0,\ldots, n-1\}$.
\end{description}
Moreover, the Ramsey-like dichotomy \textbf{(18)} implies the following facts:
\begin{prop}\label{alternative}
Given $i,j\in\{0,\ldots, n-1\}$ two distinct indices, for every $a\in R$ we have
\begin{description}
\item[(21)] $\mathscr{U}_j(g^i_a)$ either contains or is disjoint from $\{g^k_b:\, k\in\{0,\ldots, n-1\}, b\in R, b<a\}$.
\item[(22)] $\mathscr{U}_j(g^i_a)$ either contains or is disjoint from $\{g^k_b:\, k\in\{0,\ldots,n-1\}, b\in R, a<b\}$.
\end{description}
\end{prop}
\begin{proof}
Let us prove first \textbf{(21)}. Suppose that the first alternative of \textbf{(18)} holds when we take $\varepsilon=0$. Given any $b\in R, b<a,$ then $(a,b)\in A^{ij}_0$, which means that $g^i_b\in\mathscr{U}_j(g^i_a)$ and thus due to \textbf{(20)} $\{g^0_b,\ldots, g^{n-1}_b\}\subset \mathscr{U}_j(g^i_a)$. As $b$ was chosen arbitrarily, the first possibility occurs. The second case is deduced  in the same way from the second alternative of \textbf{(18)} and \textbf{(20)}. The second statement \textbf{(22)} is proved in the same way taking $\varepsilon=1$ in \textbf{(18)}.
\end{proof}
Let us consider the following definition:
\begin{defi}
Given any point $a\in R$ and indices $i,k\in\{0,\cdots, n-1\}$, we say that the function $g^i_a$ is isolated from those functions which are on its left side by the neighbourhood $\mathscr{U}_k(g^i_a)$ if 
$$\mathscr{U}_k(g^i_a)\cap\{g^l_b:\, l\in\{0,\ldots, n-1\}, b\in R, b<a\}=\emptyset.$$ 
In the same way we define when a function is isolated of those functions which are on its right side.
\end{defi}
\begin{prop}
For all but exactly one index $i\in\{0,\ldots, n-1\}$ the following holds: For every $a\in R$ the function $g^i_a$ is isolated from those which are on its left side by one of the neighbourhoods $\mathscr{U}_k(g^i_a)$. The analogous statement holds for the right side.
\end{prop}
\begin{proof}
Suppose there were $i,j$ two different indices failing the  property stated in the proposition. 
By Proposition \ref{alternative}, for every $a\in R$ and for every $l\in\{0,\ldots, n-1\}$
\begin{eqnarray*}
\mathscr{U}_{l}(g^i_a)\supset \{g^k_b:k\in\{0,\ldots, n-1\},\, b<a\},\\
\mathscr{U}_{l}(g^j_a)\supset \{g^k_b:k\in\{0,\ldots, n-1\},\, b<a\}.
\end{eqnarray*}
so in particular $\mathscr{U}_j(g^i_a)$ and $\mathscr{U}_{i}(g^j_a)$ are not disjoint for any $a\neq \min(R)$. A contradiction with \textbf{(19)}. It remains to show that at least one index $i$ must fail the statement of the proposition. Otherwise every function $g_a^i$ would be isolated from the functions on its left side. If we pick $a$ which is not isolated from the left in $R$, then since $b\mapsto h_b$ is continuous, $\{h_b : b\in R, b<a\}$ accumulates to $h_a$. By \textbf{(12)} this implies that $\{g_b^k : b\in R, b<a,k\in\{0,\ldots,n-1\}\}$ accumulates to some $g^i_a$, a contradiction. 
\end{proof}
Let $i$ and $j$ be the indices for which, for every $a\in R$, the functions $g^i_a$ and $g^j_a$ are not isolated from those which are on their left an right side respectively. As a consequence of the last proposition, we have the following two possibilities:
\begin{description}
\item[Case 1.] The indices are equal and thus, for every $a\in R$, the function $g^i_a$ is not isolated from the others by the neighbourhoods $\mathscr{U}_l(g^k_a)$ neither from the right nor from the left.
\item[Case 2.] The indices are different and thus, for every $a\in R$, the functions $g^i_a$ and $g^j_a$ are respectively isolated from those functions which are on their right and left side.
\end{description}
Let us see how \textbf{Case 1} and \textbf{Case 2} introduce, respectively, the Alexandroff $n-$plicate $D_n(2^\mathbb{N})$ and the $n-$Split Interval $S_n(I)$ inside $K$.
\\[1cm]
\textbf{Case 1.}: There exists, as we shall see immediately, a natural identification between the subspace $\{g^k_a: k\in\{0,\ldots, n-1\}, \, a\in R\}$ of $K$ and the Alexandroff $n-$plicate of the set $\{h_a:\;a\in R\}$. First of all, let us note that if we prove $D_n(\{h_a:\;a\in R\})\hookrightarrow K$ we will have proved $D_n(2^\mathbb{N})\hookrightarrow K$. Indeed, $R$ is a perfect Polish space so it contains copies of $2^\mathbb{N}$. Since the map $a\mapsto h_a$ is a homeomorphism then there exists an embedding $D_n(2^\mathbb{N})\hookrightarrow D_n(\{h_a:\;a\in R\})$ and thus $D_n(2^\mathbb{N})\hookrightarrow K$.
\\[0.2cm]
Consider the map 
$$\begin{array}{cccc}
\Phi:&D_n(\{h_a:\,a\in R\})&\longrightarrow &\{g^k_a:k\in\{0,\ldots, n-1\},\,a\in R\}\\
& (h_a,k)&\mapsto& 
g^{\psi(k)}_a

\end{array}
$$
where $\psi:\{0,\ldots,n-1\}\longrightarrow\{0,\ldots,n-1\}$ is a bijection with $\psi(0)=i$.
\\
Since all spaces involved are Rosenthal compact and hence Fr\'{e}chet-Urysohn, continuity can be checked sequentially. If $\{(h_{a_m}, k_m)\}_{m\in\mathbb{N},k_m\in\{0,\ldots,n-1\}}$ is a nontrivial convergent sequence to $(h_a,0)$, then $h_{a_m}$ converges to $h_a$ and then $a_m$ converges to $a$ since $a\mapsto h_a$ is a homeomorphism. Then, by \textbf{(12)} the sequence $\{g^{\psi(k_m)}_{a_m}\}_{m}$ must accumulate to some element of $\{g^0_a,\ldots, g^{n-1}_a\}$ but, in fact, the sequence must accumulate to $g^{i}_a$ since the other functions are isolated. Since $\Phi$ is a bijection and the domain is compact, we conclude that $\Phi$ is a homeormorphism.

\textbf{Case 2.}
Let us identify the perfect set $R\subset 2^\mathbb{N}$ with a perfect set in the unit interval $I$ through the standard embedding $2^\mathbb{N}\subset I$, that preserves the order. Now consider the map,
$$
\begin{array}{cccc}
\Psi:&S_n(R)&\longrightarrow &\{g^k_a:k\in\{0,\ldots, n-1\}, a\in R\}\\
& (a,k)&\mapsto& g^{\phi(k)}_a
\end{array}
$$
where $\phi:\{0,\ldots, n-1\}\longrightarrow\{0,\ldots,n-1\}$ is a bijection such that $\phi(0) = i$ and $\phi(1)=j$. We prove the continuity of $\Psi$ and  for the points of the form $(a,0)$ since the reasoning for the points $(a,1)$ is analogous.
Let us take a nontrivial convergent sequence $\{(a_m,k_m)\}_{m\in\mathbb{N}, k_m\in\{0,\ldots, n-1\}}$ to $(a,0)$. Then we know that $a_m$ is eventually lower than the point $a$. Arguing as in the \textbf{Case 1}, the sequence $\{g^{\phi(k_m)}_{a_m}\}_m$ must accumulate to some element of $\{g^0_a,\ldots, g^{n-1}_a\}$ and since the unique non isolated functions are $g^i_a$ and $g^j_a$ the sequence must accumulate to one of them. But, $g^j_{a}$ is isolated from the functions which are on its left side so, in fact, the sequence must accumulate to $g^i_a$ and then $\Psi$ is continuous. Again, since $\Psi$ is a bijection and the domain is compact, we conclude that $\Psi$ is a homeormorphism. Finally, since $S_n(R)\hookrightarrow K$ it is enough to apply Theorem \ref{PropertiesS_n} to obtain the desired embedding $S_n(I)\hookrightarrow K$.

\section*{Acknowledgements}

The first author was supported by MINECO and FEDER (MTM2014-54182-P) and by Fundaci\'{o}n S\'{e}neca - Regi\'{o}n de Murcia (19275/PI/14). The third author is partially supported by grants from NSERC and CNRS.\\

\textsc{Antonio Avil\'{e}s and Alejandro Poveda}

\textsc{Universidad de Murcia, Departamento de Matem\'{a}ticas, Campus de Espinardo 30100 Murcia, Spain.}

\texttt{avileslo@um.es, alejandro.poveda1@um.es}\\

\textsc{Stevo Todorcevic}

\textsc{Department of Mathematics, University of Toronto, Toronto, Canada, M5S 3G3.}

\textsc{Institut de Math\'{e}matiques de Jussieu, CNRS UMR 7586, Case 247, 4 place Jussieu, 75252 Paris Cedex, France}

\texttt{stevo@math.toronto.edu,stevo.todorcevic@imj-prg.fr}

\bibliographystyle{alpha} 
\bibliography{bibliography}

\end{document}